  \documentclass[11pt, reqno]{amsart}
\usepackage{graphicx}
\usepackage{amsmath}
\usepackage{mathrsfs}
\usepackage{enumerate}
 
\usepackage{amssymb}

\newcommand {\PP}{{I\kern-.3em P}}
\newcommand {\ZZ}{{Z\kern-.45em Z}}

\newcommand {\RR}{{\mathbb{R}}}
\newcommand {\NN}{{I\kern-.3em N}}
\newcommand {\QQ}{{\mathbb{Q}}}

\newcommand{\beq}{\begin{equation}}
\newcommand{\eeq}{\end{equation}}
\newcommand{\beqq}{\begin{equation*}}
\newcommand{\eeqq}{\end{equation*}}

 \catcode `\@=11

{\bigskip\begin{enumerate}[\bfseries 1.]}{\end{enumerate}}
{\begin{enumerate}[\bfseries a.]}{\end{enumerate}}

\newtheorem{thm}{Theorem}
\newtheorem{lemma}{Lemma}

\begin{document}
\title{The relative growth rate for partial quotients}
\author{Andrew Haas }
   \email{haas@math.uconn.edu}
  \address{University of Connecticut, Department of Mathematics,
 Storrs, CT 06269}
 
 \begin{abstract}
The rate of growth of the partial quotients of an irrational number is studied relative to the rate of approximation of the number by its convergents. The focus is on the Hausdorff dimension of exceptional sets on which different growth rates are achieved. 
  \end{abstract}

\maketitle

    In this note we look at the rate of growth of the  partial quotients $a_i$ of the irrational number
     $$x= [a_1, a_2, \ldots]= 
  \frac{1}{a_1+\frac{1}{a_2+\frac{}{a_3+ \dots}}}
$$
relative to the rate at which $x$ is approximated by its rational convergents.
 
For $x\in (0,1)$ irrational, let $\{\frac{p_n}{q_n}\}$ be the sequence of rational convergents given by the continued fraction expansion of $x$ \cite{HW}. 
It follows from  classical results of Khinchin and L\'evy \cite{bill}  that for almost all $x$

\beq\label{limits}
\lim_{n\rightarrow\infty}\frac{\log a_{n}}{n}=0\,\,\,\,\,\,
\text{and}\,\,\,\,\,\, 
\lim_{n\rightarrow\infty}-\frac{\log |x-\frac{p_n}{q_n}|}{n}= \frac{\pi^2}{6\log 2}.
\eeq
Consequently, for almost all $x\in (0,1)$
\beq\label{loga}
\lim_{n\rightarrow\infty}\frac{\log a_{n+1}}{\log  |x-\frac{p_n}{q_n}|}=
 0.
\eeq
  Here we study the Hausdorff dimension of exceptional sets on which the limit  (\ref{loga}) either  does not exist or is different from  zero. Similar, non-overlapping, problems are considered   in \cite{Kes} using more sophisticated methods of multifractal analysis.
  
  We shall write $\text{Dim}_HX$ for the Hausdorff dimension of a set $X\subset [0,1]$ and $\mathcal{H}^s(X)$ for the Hausdorff $s$-dimensional measure of $X$ \cite{falc}.
  
  Let  
    \beq\label{maindef}
  \mathscr{F}(z)=   \left \{x\in (0,1)\,  :\,    \limsup_{n\rightarrow\infty}\frac{-\log a_{n+1}}{ \log |x-\frac{p_n}{q_n}| }= z \right \}.
  \eeq
 
    \begin{thm}\label{dima}
 For $0< z\leq 1,$ 
 $ \text{Dim}_H\mathscr{F}(z)=  1-z$ and $\mathcal{H}^{1-z}(\mathscr{F}(z))=\infty.$ If $z\not\in [0,1]$ then $\mathscr{F}(z)=\emptyset.$
  \end{thm}
 
By an earlier remark, $\mathscr{F}(0)$ is a set of Lebesgue measure 1. 
 
There is an alternative characterization of the problem in terms that compare the rate of growth of the denominators of the convergents to the rate at which they approximate $x$. It is, in its own right  an interesting way to look at the problem. For $\alpha\in\RR$   define the set
\beqq
\mathscr{G}(\alpha)=\left \{x\in (0,1)\,  :\,    \limsup_{n\rightarrow\infty}\frac{\log q_n^2}{ \log |x-\frac{p_n}{q_n}| }= \alpha \right \}.
  \eeqq

\begin{lemma}
When $z\in [0,1]$, $\mathscr{F}(z)=\mathscr{G}(z-1).$
\end{lemma}
\begin{proof}
Define the approximation constants $\theta_n(x)=q_n|q_nx-p_n|$. From the classical theory of continued fractions we have
\beq\label{theta}
\theta_n(x)=\frac{1}{a'_{n+1}+\frac{q_{n-1}}{q_n}}
\eeq
where $a'_{n+1}=a_{n+1}+[a_{n+2},\ldots]$ \cite{HW}.
Therefore 
\beqq
\limsup_{n\rightarrow\infty}\frac{-\log a_{n+1}}{ \log |x-\frac{p_n}{q_n}| }=
 \limsup_{n\rightarrow\infty}\frac{\log \theta_n(x)}{\log  |x-\frac{p_n}{q_n}| }=
 \limsup_{n\rightarrow\infty}\frac{\log q^2_n}{\log  |x-\frac{p_{n }}{q_{n }}| }+1 . 
  \eeqq
  \end{proof}
  
  At this point it is an easy matter to show that $\mathscr{F}(1)$ is an infinite set and therefore $\mathcal{H}^0(\mathscr{F}(1))=\infty.$ In fact, if one chooses the partial quotients so that $q_n^{2n}<a_{n+1},$ then using (\ref{theta}) it follows that the limit in (\ref{loga}) is equal to 1.
 In order to complete the proof of Theorem \ref{dima}, we shall work with the alternative formulation suggested by the lemma and prove
 \begin{thm}\label{g}
 For $\alpha\in (-1,0]$, $\text{Dim}_H\mathscr{G}(\alpha)=|\alpha|$ and  $\mathcal{H}^{|\alpha|}(\mathscr{G}(\alpha)) =\infty.$ If $\alpha\not\in [-1,0]$, then $\mathscr{G}(\alpha)=\emptyset.$
 \end{thm}

 The set  $  \mathscr{G}(-1)=\mathscr{F}(0)$ has lebesgue measure 1. Interestingly, the results in \cite{Kes} imply that if the second limit in (\ref{limits}) exists for a number $x$ (not necessarily taking the value given in  (\ref{limits})), then $x\in \mathscr{G}(-1).$ 

The last sentence of Theorem \ref{g} is  elementary and is a consequence of the following basic property of the convergents  \cite{HW} 
 \beqq
 |x-\frac{p_n}{q_n}|<\frac{1}{q_n^2}.
 \eeqq

 The main tool in the proof of Theorem \ref{g} is Jarnik's "zero-infinity" law \cite{berdicvel, bervel, jarnik}. We need to establish some notation and reframe the problem so that Jarnik's Theorem will apply.
 
The abbreviation FIM will be used in place of  the phrase, "for  infinitely many." Given $\tau\in(-1,0)$, and $0 \leq \epsilon<|\tau|$, define   
 \beqq
 \psi_{(\tau,\epsilon)}(r)=r^{\frac{2}{\tau+\epsilon}}.
 \eeqq
Consider the related equation
 \beq\label{mainineq}
 |x-\frac{p}{q}|< \psi_{(\tau,\epsilon)}(q)=q^{\frac{2}{\tau+\epsilon}}
 \eeq
 and the set 
 \beqq
 W( \psi_{(\tau,\epsilon)})= \left  \{x\in [0,1]:   \,  |x-\frac{p}{q}|< q^{\frac{2}{\tau+\epsilon}} \,\,\,   \text{FIM }\,\, \frac{p}{q}\in\QQ\right  \}. 
\eeqq
 
 We are not interested in just any rationals but rather in the convergents. 
Define
 \beqq 
W^*( \psi_{(\tau,\epsilon)}) =\left \{x\in [0,1]\, :\,  |x-\frac{p_n}{q_n}|< q_n^{\frac{2}{\tau+\epsilon}}   \,\, \text{FIM convergents}\, \frac{p_n}{q_n} \, \text{of}\, x\right \}.
\eeqq
 Observe that when $q$ is sufficiently large,
\beq\label{q}
 q^{\frac{2}{\tau+\epsilon}}<\frac{1}{2}q^{-2}.
 \eeq
If $\frac{p}{q}$ satisfies inequalities (\ref{mainineq}) and (\ref{q}) then it is a convergent of $x$ \cite{HW}.
Therefore, except for finitely many rationals,  $\frac{p}{q}$ satisfies (\ref{mainineq}) if and only if  it is a convergent of $x$. Consequently, $W( \psi_{(\tau,\epsilon)})=W^*( \psi_{(\tau,\epsilon)}).  $

Combining the last observation with  a simple manipulation of equation (\ref{mainineq})  yields
 \beqq\label{name} 
W( \psi_{(\tau,\epsilon)}) =\left \{x\in [0,1]\,:\,  \frac{\log q^2_n}{\log  |x-\frac{p_{n }}{q_{n }}| }>\tau+\epsilon\,\,\,   \text{FIM convergents}\,\,  \frac{p_n}{q_n}\,\, \text{of}\,\, x\right \}.
\eeqq
It is therefore clear that for $-1<\tau<\alpha\leq 0$
\beq\label{sub}
 W(\psi_{(\tau,0)})\supset \mathscr{G}(\alpha).
 \eeq

Now we turn to the computation of Hausdorff dimension.
 
\begin{lemma}\label{prop}
For any $\tau\in (-1,0),\, \text{dim}_HW(\psi_{(\tau,0)})=|\tau|$, $\mathcal{H}^{|\tau|}(W(\psi_{(\tau,0)}))=\infty$ and for $\epsilon>0,\, \mathcal{H}^{|\tau|}(W( \psi_{(\tau,\epsilon)}))=0$

\end{lemma}

\begin{proof}
If $ \psi:\RR^+\rightarrow \RR^+$ is a decreasing function, then a basic version of Jarnik's Theorem \cite{bervel} says that   for $s\in [0,1)$
\beqq
\mathcal{H}^s(W( \psi_{(\tau,\epsilon)} ))= 
\begin{cases}
0 & \text{if}\,\,\,\, \sum_{r=1}^{\infty}r( \psi(r ) )^s<\infty\\
\infty & \text{if}\,\,\,\, \sum_{r=1}^{\infty}r( \psi(r ) )^s=\infty.
\end{cases} 
\eeqq 
When $\psi(r)= \psi_{(\tau,\epsilon)}$, the series' involved are easy to analyze and  it follows that for  $\tau\in (-1,0)$ and $0 \leq \epsilon<|\tau|$,
\beqq
 \mathcal{H}^{s}(W(\psi_{(\tau,\epsilon)}))=
\begin{cases}
  0&\text{for $s>-\tau-\epsilon.$ } \\
\infty &\text{for $s\leq-\tau-\epsilon$}
\end{cases}
\eeqq

From this we conclude that $\text{dim}_HW(\psi_{(\tau,0)})=|\tau|$ and moreover, that  $ W(\psi_{(\tau,0)})$   has infinite $|\tau|$-measure. Also, when  $\epsilon >0$ the sets $W(\psi_{(\tau,\epsilon)})$  have $|\tau|$-measure zero.
\end{proof}
 
\noindent {\em Proof of Theorem \ref{g}}.\, 
First, it follows from the inclusion (\ref{sub}) and Lemma \ref{prop} that 
\beq\label{first}
  \text{dim}_H\mathscr{G}(\alpha)\leq \text{dim}_HW(\psi_{(\tau,0)})=|\tau|
  \eeq
  for all  $-1<\tau<\alpha\leq 0$. In particular, this gives   $\text{dim}_H\mathscr{G}(0)=0$

 Now suppose $\alpha\in (-1,0)$ and pick $k<0$ so that $\frac{1}{k}<|\alpha|$. Define the set 
\beqq
E(\alpha)= W(\psi_{(\alpha,0)})\setminus\left [\bigcup_{n=k}^{\infty} W(\psi_{(\alpha,\frac{1}{n})})\right ].
\eeqq
It is clear that  $  E(\alpha)\subset \mathscr{G}(\alpha).$  Furthermore, applying Lemma \ref{prop},  we see that $\text{dim}_HE(\alpha)=|\alpha|$ and $\mathcal{H}^{|\alpha|}(E(\alpha))=\infty.$ Thus,
\beq\label{second}
|\alpha|=\text{dim}_HE(\alpha)\leq  \text{dim}_H\mathscr{G}(\alpha).
\eeq 
 Together equations (\ref{first}) and (\ref{second})
  allow us to conclude that $\text{dim}_H\mathscr{G}(\alpha)=|\alpha|.$ Since $E(\alpha)$ has infinite $|\alpha|$-measure, so must the larger set $\mathscr{G}(\alpha)$ 
  \qed

\section*{Acknowledgements}    
I am very grateful to the referee, whose feedback and advice resulted in a far stronger paper.

\end{document}